\documentclass[11pt]{amsart}
\usepackage{latexsym,amssymb,amsmath,youngtab}
\usepackage{youngtab}
\usepackage{epsfig}
\usepackage{amsmath,amsthm,amssymb,amscd,MnSymbol}
\usepackage[mathscr]{eucal}
\usepackage{verbatim}
\usepackage{tikz}
\usepackage{caption}
\usepackage{subcaption}
\usepackage[arrow,matrix,graph,frame,poly,arc,tips]{xy}

\usepackage{color}
\newtheoremstyle{custom}
  {3pt}
  {3pt}
  {\slshape}
  {}
  {\bfseries}
  {.}
  { }
   {}
\theoremstyle{custom}
\newtheorem{theorem}{Theorem}[section]

\newtheorem{proposition/definition}[theorem]{Proposition/Definition}

\newtheorem{corollary}[theorem]{Corollary}
\newtheorem{conjecture}[theorem]{Conjecture}

\theoremstyle{definition}

\newtheorem{example}[theorem]{Example}

\theoremstyle{remark}
\newtheorem{remark}[theorem]{Remark}

\newtheoremstyle{exercise}
  {3pt}
  {6pt}
  {}
  {}
  {\bfseries}
  {:}
  { }
   {}
\theoremstyle{exercise}
\newtheorem{exercise}[theorem]{Exercise}
\newtheoremstyle{exercises}
  {3pt}
  {6pt}
  {}
  {}
  {\bfseries}
  {:}
  {\newline}
   {}

\theoremstyle{exercise}
\newtheorem{exercises}[theorem]{Exercises}

\input epsf
\def\boxit#1{\vbox{\hrule height1pt\hbox{\vrule width1pt\kern3pt
  \vbox{\kern3pt#1\kern3pt}\kern3pt\vrule width1pt}\hrule height1pt}}

\def\trank{\text{rank}}

\def\BC{\mathbb C}

\def\tdim{{\rm dim}}

\def\hd{,...,}

\def\TH{\Theta}

\def\11{\mathbf 1}

\def\o{\omega}

\def\s{\sigma}

\def\d{\delta}

\def\ot{{\mathord{ \otimes } }}

\def\ra{{\mathord{\;\rightarrow\;}}}

\def\dim{{\rm dim}\;}

\def\ep{\epsilon}

\def\s{\sigma}
\def\t{\tau}

\def\TH{\Theta}

\def\FS{\mathfrak  S}

\def\ol{\overline}

\def\BC{\mathbb  C}

\def\ep{\epsilon}

\def\ci{\mathcal  I}

\def\hd{, \hdots ,}

\def\ra{\rightarrow}

\def\tdet{\operatorname{det}}

\def\tperm{\operatorname{perm}}

\def\tend{\operatorname{End}}

\def\tdim{\operatorname{dim}}

\def\trank{\operatorname{rank}}

\def\be{\begin{equation}}
\def\ene{\end{equation}}

\DeclareMathOperator{\tlog}{log}
\DeclareMathOperator\tspan{span}

\def\vp{{\bold V\bold P}}
\def\vnp{{\bold V\bold N\bold P}}

\def\tspan{{\rm span}}

\newcommand{\Id}{\operatorname{Id}}

\def\trank{{\mathrm {rank}}}

\def\Th{\Theta}

\def\tln{\operatorname{ln}}

\begin{document}
\title[Shifted partial derivatives cannot separate]%
{The method of shifted partial derivatives cannot separate the permanent from the determinant}

\author{Klim Efremenko}
\address{Simons Institute for Theoretical Computing, Berkeley}
\email{klimefrem@gmail.com}

\author{J.M. Landsberg}
\thanks{Landsberg supported by NSF    DMS-1405348}
\address{Department of Mathematics, Texas A\&M University}
\email{jml@math.tamu.edu}

\author{Hal Schenck}
\thanks{Schenck supported by NSF DMS-1312071}
\address{Department of Mathematics, University of Illinois}
\email{schenck@math.uiuc.edu}

\author{Jerzy Weyman}
\thanks{Weyman supported by NSF DMS-1400740}
\address{Department of Mathematics, University of Connecticut}
\email{jerzy.weyman@gmail.com}


\begin{abstract}
 The method of shifted partial
 derivatives   
  introduced in
 \cite{DBLP:journals/eccc/Kayal12,gupta4}  was used to prove a super-polynomial lower bound
 on the size of depth four circuits needed to compute the permanent.
 We show that this method alone cannot   prove that the padded permanent 
 $\ell^{n-m}\tperm_m$ cannot
 be realized inside the $GL_{n^2}$-orbit closure of the determinant $\tdet_n$
 when $n>2m^2+2m$. Our proof relies on several simple degenerations
 of the determinant polynomial, Macaulay's theorem that gives a lower bound on the
 growth of an ideal,  and a lower bound estimate from \cite{gupta4}
 regarding the shifted partial derivatives of the determinant.
\end{abstract}
 
\maketitle

\section{Introduction}

Let $\FS_m$ denote the permutation group on $m$ elements and let 
$y^i_j$ be linear coordinates on $\BC^{m^2}$. The permanent polynomial is
$$
\tperm_m(y^i_j)=\sum_{\s\in \FS_m} y^1_{\s(1)}\cdots y^m_{\s(m)}.
$$
Valiant's famous conjecture $\vp\neq\vnp$ may be phrased as:

\begin{conjecture}\label{valconb} \cite{vali:79-3} There does not
exist a polynomial size circuit computing the permanent.
\end{conjecture}

Let $W=\BC^N$ with linear coordinates $x_1\hd x_N$, let $W^*$ denote the dual vector space,
let $S^nW$ denote the space of degree $n$ homogeneous
polynomials on $W^*$, and let $Sym(W)=\bigoplus_nS^nW$.
Let $\tend(W)$ denote the space of endomorphisms of $W$, so in particular if
$P\in S^nW$, $\tend(W)\cdot P\subset S^nW$ is the set of homogeneous degree $n$
polynomials obtainable by linear specializations of the variables $x_1\hd x_N$   in $P(x_1\hd x_N)$.

Since the determinant $\tdet_n\in S^n\BC^{n^2}$ is in $\vp$, Conjecture \ref{valconb} would imply the following conjecture:

\begin{conjecture}\label{valrephrase}  \cite{vali:79-3}
Let $\ell$ be a linear coordinate on $\BC^1$ and consider any linear inclusion
$\BC^1\oplus \BC^{m^2}\ra W=\BC^{n^2}$, so in particular $\ell^{n-m}\tperm_m\in S^nW$.
Let $n(m)$ be a polynomial. Then for all sufficiently large $m$, 
$$
[\ell^{n-m}\tperm_m]\not\in  \tend(W) \cdot [\tdet_{n(m)}] .
$$
\end{conjecture}

The polynomial $\ell^{n-m}\tperm_m$ is called the {\it padded permanent}.

Instead of arbitrary circuits, by  \cite{MR0660280,
DBLP:journals/eccc/GuptaKKS13,MR3303254,koirand4,AgrawalVinay} one 
could prove Valiant's conjecture by  restricting
to depth-four circuits 
and proving a stronger lower bound:
If  $P_n\in S^n\BC^N$  is a 
sequence of polynomials  that can be computed by
a circuit of size $s=s(n)$, then
$P_n$ is   computable by a homogeneous $\Sigma\Pi\Sigma\Pi$ circuit of   size $2^{O(\sqrt{n \tlog(ns)\tlog(N)})}$.
So to prove $\vp\neq\vnp$,
it would be sufficient to show the permanent $\tperm_m$ is not computable by a size $2^{O(\sqrt{m \tlog(  poly(m)) })}$
homogeneous $\Sigma\Pi\Sigma\Pi$ circuit.
Work of   Gupta, Kamath, Kayal, and
 Saptharishi \cite{gupta4}   generated considerable excitement,  because it came 
 tantalizingly close to proving Valiant's conjecture by showing that the permanent
does not admit a size $2^{o(\sqrt{m})}$ homogeneous $\Sigma\Pi\Sigma\Pi$ circuit with bottom fanin bounded
by $\sqrt{m}$.

Any method of proof that separates $\vp$ from $\vnp$ would
also have to separate the determinant from the permanent. We show that this
cannot be done with the method of   proof in \cite{gupta4},  the {\it method of shifted partial derivatives}.
This method builds upon
the {\it method of partial derivatives} (see, e.g.,  \cite{MR1486927}), which dates back to Sylvester
\cite{sylvestercat}.

\subsection{The method of shifted partial derivatives}
The space
$S^kW^*$ may be interpreted as  the space of homogeneous differential operators on $Sym(W)$
of order $k$.
Given a homogeneous polynomial
$P\in S^nW$, consider the linear map
\begin{align*}
P_{k,n-k}: S^kW^*&\ra S^{n-k}W\\
D&\mapsto D(P).
\end{align*}
In coordinates  the map is $\frac{\partial^k}{\partial x_{i_1}\cdots 
\partial x_{i_k}}\mapsto \frac{\partial^kP}{\partial x_{i_1}\cdots 
\partial x_{i_k}}$.

    Given polynomials $P,Q\in S^nW$, and $k<n$, 
   $P\in  \tend(W)\cdot Q$ implies that $\trank(P_{k,n-k})\leq \trank(Q_{k,n-k})$.
The method of partial derivatives is to find a $k$ such that 
$\trank(P_{k,n-k})> \trank(Q_{k,n-k})$ to prove $P\not\in \tend(W)\cdot Q$.
   
Now consider $P_{k,n-k}\ot \Id_{S^{\t}W}: S^kW^*\ot S^{\t}W \ra S^{n-k}W\ot S^{\t}W$ and 
project (multiply) the image to $S^{n-k+\t}W$  to obtain a map
\begin{align*}
P_{(k,n-k)[\t]}: S^kW^*\ot S^{\t}W&\ra S^{n-k+\t}W\\
D\ot R&\mapsto D(P)R
\end{align*}
Again 
 $P\in  \tend(W)\cdot Q$ implies that $\trank(P_{(k,n-k)[\t]})\leq \trank(Q_{(k,n-k)[\t]})$. 
 The method of shifted partial derivatives is to find   $k,\t$ such that 
$\trank(P_{(k,n-k)[\t]})> \trank(Q_{(k,n-k)[\t]})$ to prove $P\not\in \tend(W)\cdot Q$.

\begin{remark} Both these methods are {\it algebraic} in the sense that they actually
prove $P\not\in \ol{\tend(W)\cdot Q}$ where the overline denotes Zariski closure.
Most known lower bound techniques for Valiant's conjecture have this property, see
\cite{GrochowGCTUnity}.
\end{remark}

\begin{remark} These methods may be viewed as special cases of the {\it  Young-flattenings}
introduced in \cite{MR3081636}.
\end{remark}

From the perspective of
algebraic geometry,  the method of shifted partial derivatives  compares  growth of {\it Jacobian ideals}:
For $P\in S^nW$, consider the 
ideal in $Sym(W)$  generated by the partial
derivatives of $P$ of order $k$. Call this the {\it $k$-th Jacobian ideal} of $P$,
and denote it by  $\ci^{P,k}$. It is generated in degree $n-k$.
The method is comparing the dimensions of the Jacobian ideals
in degree $n-k+\t$, i.e., the  {\it Hilbert functions} of the Jacobian ideals.

\subsection{Statement of the result}
 We prove this method cannot 
 give better than a quadratic separation of  the permanent from the determinant:
 
\begin{theorem}\label{badnews} There exists a constant $M$ such that for all $m>M$ and every
$n>2m^2+2m$, any $\t$ and any $k<n$, 
$$\trank((\ell^{n-m} \tperm_m)_{(k,n -k)[\t]})  <\trank((\tdet_{n})_{( k,n  -k )[\t]}).
$$
\end{theorem}

Despite this, it may be possible that a more general Young flattening is able to prove, e.g. a
$\o(m^2)$ lower bound on $n$. This motivated the companion paper \cite{ELSWmfr} where we study Jacobian
ideals and their minimal free resolutions.

\subsection{Overview of the proof}
The proof of Theorem \ref{badnews}  splits into four cases:
\begin{itemize}
\item (C1) Case $k>n-\frac n{m+1}$,
\item (C2) Case $2m\leq k\leq n-2m$,
\item (C3) Case $k<2m$ and $\t> \frac 32 n^2m$,
\item (C4) Case $k<2m$ and $\t<6\frac {n^3}m$.
\end{itemize}
Note that C1,C2 overlap when $n>2m^2+2m$ and 
C3,C4 overlap when $n>\frac{m^2}4$, so it suffices to take $n>2m^2+2m$.

\medskip

  In the first case, the proof has nothing to do with
the padded permanent or its derivatives: it is valid for any   polynomial in $m^2+1$ variables.
Cases 2,3 only use that we have a padded polynomial. 
In the fourth case, the only property of the permanent that is used is an estimate on the size of
the space of its partial derivatives. Case C1 is proved by showing that in this range the partials of the determinant can be degenerated
into the space of {\it all} polynomials of degree $n-k$ in $m^2+1$ variables.  
Cases C2,C3  use  that when $k<n-m$,    the Jacobian ideal of  {\it any} padded polynomial $\ell^{n-m}P\in S^nW$ is contained in
the ideal generated in degree $n-m-k$  by $\ell^{n-m-k}$, which has slowest possible growth
by Macaulay's theorem as explained below. Case C2 compares that ideal with the Jacobian
ideal of the determinant; it is smaller in degree $n-k$ and therefore smaller
in all higher degrees by Macaulay's theorem.
Case C3 compares that ideal with an ideal with just two generators in degree $n-k$.
Case C4 uses a lower bound for the determinant used in \cite{gupta4} and 
  compares it with a very crude upper bound for the dimension of the space of shifted partial
  derivatives for the permanent.

    We first review Macaulay's theorem and then prove each case.
    
    \medskip

We will use the   notation:
\begin{align*}
t&=\t+n-k.
\end{align*}
Fix index ranges $1\leq s,t,\leq n$, and $1\leq i,j\leq m$. 
If $\ci\subset Sym(W)$ is an ideal, we let $\ci_d\subset S^dW$ denote its component in degree $d$. 

We make repeated use of the estimate 
\be\label{lnnfac}
\tln(q!)=q\tln(q)-q +\TH (\tln(q)).
\ene

\subsection{Acknowledgments}  Efremenko, Landsberg and Weyman thank
the Simons Institute for the Theory of Computing, UC Berkeley, for providing a wonderful 
environment during the fall 2014 program {\it Algorithms and Complexity in Algebraic Geometry} to work on  this article.
We also thank the anonymous referee of a previous version  for pointing out a gap in our argument.

\section{Macaulay's Theorem}
We only use Corollary \ref{maccor} from this section in the proof of Theorem \ref{badnews}.

\begin{theorem} [Macaulay,    see, e.g.,  \cite{MR1648665}] \label{macthm}
Let $\ci\subset Sym(W)$ be an ideal, let  $d$
be a natural number, and set $Q=\tdim S^{d}V/\ci_{d}$. Write 
\be\label{qex}
Q=\binom {a_{d}}{d}+\binom{a_{{d}-1}}{{d}-1}+\cdots +\binom{a_\d}{\d}
\ene
with $a_{d}>a_{{d}-1}>\cdots >a_{\d}$ (such an  expression exists and is unique).
Then
\be\label{macbnd}
  \tdim\ci_{d +\t}\geq 
\binom{N+d+\t-1}{d+\t}
-\left[  \binom {a_{d}+\t}{{d}+\t}+\binom{a_{{d}-1}+\t}{{d}+\t-1 }
+\cdots +\binom{a_\d+\t}{\d+\t} \right].
\ene
\end{theorem}

\begin{remark} Gotzman \cite{MR0480478} showed that if $\ci$ is generated in degree at most $d$,
then equality  is achieved for all $\t$
in \eqref{macbnd} if equality holds for $\t=1$.
Ideals  satisfying this minimal growth exist, they
are {\it lex-segment ideals}, see  \cite{MR1648665}.
\end{remark}

\begin{remark}Usually Macaulay's theorem is stated in terms of the coordinate
ring $\BC[X]:=Sym(W)/\ci$ of the variety (scheme) $X\subset   W^*$ that is the zero set of $\ci$, namely
$$
\tdim \BC[X]_{d+\t}\leq 
\binom {a_{d}+\t}{{d}+\t}+\binom{a_{{d}-1}+\t}{{d}+\t-1 }
+\cdots +\binom{a_\d+\t}{\d+\t}.
$$
\end{remark}

\begin{corollary}\label{maccor} Let $W=\BC^N$.
Let $\ci$ be an ideal such that $\tdim \ci_d\geq \tdim S^{d-q}W=\binom{N+d-q-1}{d-q}$ for
some $q< d$. Then $\tdim \ci_{d+\t}\geq \tdim S^{d-q+\t}W=\binom{N+\t+d-q-1}{\t+d-q}$.
\end{corollary}

\begin{proof}
First use the identity
\be\label{binomid}
\binom{a+b}{b}=\sum_{j=1}^q\binom{a+b-j}{b-j+1}+\binom{a+b-q}{b-q}
\ene
with  $a=N-1$, $b=d$.
Write this as 
$$
\binom{N-1+d }{d }= Q_{d} + \binom{N-1+d-q }{d-q }.
$$
Set
$$
Q_{d+\t}:= \sum_{j=1}^q\binom{N-1+d +\t-j}{d +\t-j+1}.
$$
By Macaulay's theorem, any ideal $\ci$ with 
$$
\tdim \ci_{d}\geq \binom{N-1+d-q }{d-q }
$$
  must satisfy 
$$
\tdim \ci_{ d+\t}\geq \binom{N-1+d+\t }{d +\t}- Q_{d+\t}=\binom{N-1+d-q+\t}{d-q+\t}.
$$
\end{proof}

We will use Corollary \ref{maccor} with $N=n^2$, $d=n-k$, and $d-q=m$.

\section{Case C1}
Our assumption is $(m+1)(n-k)<n$. 
It will be sufficient to show that some $R\in \tend(W)\cdot \tdet_n$ satisfies
$\trank((\ell^{n-m} \tperm_m)_{(k,n -k)[\t]})  <\trank(R_{k,n -k[\t]}) 
$.
Block the matrix $x=(x^s_t)\in \BC^{n^2}$ as a union of $n-k$ $m\times m$ blocks
in the upper-left corner plus the remainder, which by our assumption includes at least $n-k$ elements on the diagonal.
Set each   diagonal block to  the matrix $(y^i_j)$
(there are $n-k$ such blocks), fill the remainder of the diagonal with $\ell$
(there are at least $n-k$ such terms),
and fill the remainder of the matrix with zeros.
Let $R$ be the restriction of the determinant to this subspace.
Then   the space of partials of $R$ of degree $n-k$, 
  $R_{k,n-k}(S^n\BC^{n^2 *})\subset S^{n-k}\BC^{n^2}$ contains
  a space isomorphic to $S^{n-k}\BC^{m^2+1}$, and  $\ci^{\ell^{n-m}\tperm_m,k}_{n-k}\subset S^{n-k}\BC^{m^2+1}$ so we conclude.

\begin{example}Let $m=2$, $n=6$. 
The  matrix is
$$
\begin{pmatrix}
y^1_1 & y^1_2 & & & & \\
y^2_1 & y^2_2 & & & & \\
& & y^1_1 & y^1_2 & &   \\
& & y^2_1 & y^2_2 & &   \\
& &   &   &\ell &   \\
& &   &   & &\ell  
\end{pmatrix} .
$$
 The polynomial $(y^1_1)^2$ is the image
of 
$\frac{\partial^4}{\partial x^2_2\partial x^4_4\partial x^5_5\partial x^6_6}$
and the polynomial $y^1_2y^2_2$ is the image of $\frac{\partial^4}{\partial x^2_1\partial x^3_3\partial x^5_5\partial x^6_6}$.
\end{example}

\section{Case C2}
As long as $k<n-m$, $\ci^{\ell^{n-m}\tperm_m, k}_{n-k}\subset \ell^{n-m-k}\cdot S^mW$,
so 
\be\label{permesta}
\tdim \ci^{\ell^{n-m}\tperm_m, k}_{n-k+\t}\leq\binom{n^2+m+\t-1}{m+\t }.
\ene
By Corollary \ref{maccor}, it will be sufficient to show that
\begin{equation} \label{eq:2} 
\tdim \ci^{\tdet_{n},k}_{n-k}={n\choose k}^2\geq \tdim S^m W={n^2+m-1 \choose m}.
\end{equation}
In the range  $2m \leq k \leq n-2m$, the quantity $n\choose k$ is minimized at $k=2m$ and $k=n-2m$, so     it is enough to show that 
\be\label{esta}
{n\choose 2m}^2  \geq  {n^2+m-1 \choose m}.
\ene  
Using \eqref{lnnfac}
\begin{align*}
 \tln{n\choose 2m}^2&= 2[n\tln(n)-2m\tln(2m) -(n-2m)\tln(n-2m)-\Th(\tln(n))\\
&=2[n\tln(\frac{n}{n-2m}) +2m\tln(\frac{n-2m}{2m})] -\Th(\tln(n))\\
&=-4m+ m\tln(\frac{n}{2m}-1)^4-\Th(\tln(n)),
\end{align*}
 where to obtain the last line we used $(1-\frac{2m}{n})^{n} > e^{-2m}e^{\Th(\frac{m^2}n)}$, and
\begin{align*}
 \tln {n^2+m-1 \choose m}&= (n^2+m-1)\tln(n^2+m-1) -m\tln(m) - (n^2-1)\tln(n^2-1)-\Th(\tln(n))\\
&= (n^2-1)\tln(\frac{n^2+m-1}{n^2-1} + m\tln(\frac{n^2+m-1}{m})-\Th(\tln(n))\\
&= m\tln(\frac{n^2}m-\frac{m-1}m)+m-\Th(\tln(n)).
\end{align*}
So \eqref{esta} will hold when $(\frac{n}{2m}-1)^4> \tln(e^5) (\frac{n^2}m-\frac{m-1}m)$ which holds for all sufficiently large
$m$ when $n>m^2$.

\section{Case  C3}
Here we simply degenerate $\tdet_n$  to $R=\ell_1^n+\ell_2^n$ by e.g., setting all diagonal elements to $\ell_1$,
all the sub-diagonal elements to $\ell_2$ as well as the $(1,n)$-entry, and setting all other elements of the matrix
to zero.
Then  $\ci^{R,k}_{n-k}=\tspan\{\ell^{n-k}_1,\ell^{n-k}_2\}$.
In degree $n-k+\t$, this ideal consists of
all polynomials of the form $\ell_1^{n-k}Q_1+\ell_2^{n-k}Q_2$ with $Q_1,Q_2\in S^\t \BC^{n^2}$,
which has dimension $2\dim S^\t\BC^{n^2}- \tdim S^{\t-(n-k)}\BC^{n^2}$ because the polynomials
of the form $\ell_1^{n-k}\ell_2^{n-k}Q_3$ with $Q_3\in S^{\t-(n-k)}\BC^{n^2}$ appear in both terms.
By this discussion, or simply because   this is a complete intersection ideal,  we have
 
\be\label{degentwo}
\tdim \ci^{R, k}_{n-k+\t} = 2\binom{n^2+ \t  -1}{\t }-\binom{n^2+\t- (n-k)-1}{\t- (n-k)}.
\ene

We again use the estimate \eqref{permesta} from Case C2, so we need to show
$$
2\binom{n^2+ \t-1}{\t}
  - \binom{n^2+\t+m-1}{\t+m}-\binom{n^2+\t-  (n-k) -1}{\t-  (n-k) }>0.
$$
Divide by $\binom{n^2+ \t-1}{\t}$. We need
\begin{align}
2>&\Pi_{j=1}^{m }\frac{n^2+\t+m-j}{\t+m-j} +\Pi_{j=1}^{n-k }\frac{\t-j}{n^2+\t-j}\\
&= \Pi_{j=1}^{m }(1+ \frac{n^2}{\t+m-j}) +\Pi_{j=1}^{n-k}(1- \frac{n^2}{n^2+\t-j})
\end{align}
The second line is less  than 
\be\label{star}
(1+\frac{n^2}{\t })^m +(1- \frac{n^2}{n^2+\t -1})^{n-k}.
\ene
We analyze \eqref{star} as a function of $\t$.  
Write $\t=n^2m\d$, for some constant $\d$. 
Then \eqref{star} is bounded above by 
$$
e^{\frac 1\d} + e^{\frac 2\d -\frac n {m\d} }.
$$
The second term goes to zero for large $m$, so we   just
need   the first term to be less than $2$, so we take, e.g. $\d=\frac 32$.
 
\section{Case C4}
We use a lower bound on $\ci^{\tdet_n,k}_{n-k+\t}$ from \cite{gupta4}:
Given a polynomial $f$ given in coordinates, its {\it leading monomial} 
in some order, is the monomial in its expression that is highest in the order.
If an ideal is generated by $f_1\hd f_q$ in degree $n-k$, then in degree $n-k+\t$, its
 dimension is at least the number of monomials in degree $n-k+\t$ that contain
a leading monomial from one of the $f_j$.

If we order the variables in $\BC^{n^2}$ by $x^1_1>x^1_2> \cdots > x^1_n>x^2_1>\cdots >x^n_n$,
then the leading monomial of any minor is the product of the elements on the principal
diagonal. Even   this  is difficult to estimate, so in \cite{gupta4} they restrict
further to only look at leading monomials among the variables
on the diagonal and super diagonal: $\{ x^1_1\hd x^n_n, x^1_2,x^2_3\hd x^{n-1}_n\}$.
Among these, they compute that the number of leading monomials of degree $n-k$  is
$\binom{n+k}{2k}$.
Then then show that in degree $n-k+\t$ the dimension of this ideal is bounded below
by $\binom{n+k}{2k}\binom{n^2+\t-2k}{\t}$ so we conclude
\be\label{permlowershift}
\tdim \ci^{\tdet_n,k}_{n-k+\t}\geq
\binom{n+k}{2k}\binom{n^2+\t-2k}{\t}.
\ene

We compare this with the very crude estimate 
$$  \tdim \ci^{\ell^{n-m}\tperm_m,k}_{n-k+\t}
\leq \sum_{j=0}^k{\binom mj}^2\binom{n^2+\t-1}{\t},
$$
where $\sum_{j=0}^k{\binom mj}^2$ is the dimension of the space
of partials of order $k$ of $\ell^{n-m}\tperm_m$, and 
the $\binom{n^2+\t}{\t}$ is what one would have if there were
no syzygies (relations among the products).

We have
\begin{align}
\label{ea} \tln\binom{n+k}{2k}&=n\tln\frac{n+k}{n-k} +k\tln\frac{n^2-k^2}{4k^2}+\Th(\tln(n))\\
\nonumber &= k\tln\frac{n^2-k^2}{4k^2}+\Th(\tln(n))\\
\label{ec}
\tln \frac{ \binom{n^2+\t-2k}{ \t} }{  \binom{n^2+\t-1}{\t} }
&= n^2\tln\frac{(n^2+\t-2k)(n^2-1)}{(n^2-2k)(n^2+\t-1)}
+\t\tln\frac{n^2+\t-2k}{n^2+\t-1} +2k\tln\frac{n^2-2k}{n^2+\t- 2k} +\Th(\tln(n))\\
\nonumber
&= -2k\tln(\frac{\t}{n^2}+1) +\Th(\tln(n)),
\end{align}
where the second lines of expressions \eqref{ea},\eqref{ec} hold because
  $k<2m$. 
We split this into two sub-cases: $k\geq \frac m2$ and $k< \frac m2$.

\subsection{Subcase $k\geq \frac m2$}
In this case we have  
$\sum_{j=0}^k{\binom mj}^2<\binom {2m}m$.
We show 
the ratio  
\be\label{trivhavec}
\frac{\binom{n+k}{2k}\binom{n^2+\t-2k}{ \t}}
{ 
 \binom {2m}m \binom{n^2+\t-1}{\t}
} 
\ene
is greater than one. 
Now
\begin{align}\label{ebb}
 \tln   \binom {2m}m  &= m\tln 4 +\Th(\tln(m)). \\
 \nonumber 
 &= k\tln 4^{\frac mk}  +\Th(\tln(m)).
\end{align}
Then \eqref{trivhavec} is greater than one if 
$$
k\tln\left( \frac{n^2-k^2}{4k^2}\frac 1{ (\frac{\t}{n^2}+1)^2} \frac 1{4^{\frac mk}}\right) 
   \pm   \Th(\tln(n))
$$
is positive. 
This will occur if
$$
\frac{n^2-k^2}{4k^2}\frac 1{ (\frac{\t}{n^2}+1)^2} \frac 1{4^{\frac mk}}>1
$$
i.e.,
if
$$
\t < n^2(\frac{\sqrt{n^2-k^2}}{2k 4^{\frac m{2k}}}-1).
$$
Write this   as
\be\label{estbo}
\t<n^2(\frac{n }{2\ep m4^{\frac 1{2\ep}}}-1).
\ene
The worst case is $\ep=2$ where it suffices to take 
$\t<\frac{n^3}{ 6m}$.

\subsection{Subcase $k<\frac m2$}
Here we use that
$\sum_{j=0}^k{\binom mj}^2<k {\binom {m}k}^2$ and the same argument gives
that it suffices to have
$$
\t < n^2(\frac{\sqrt{n^2-k^2}}{2k} \frac 1{\sqrt{\frac mk}-1}-1).
$$
The worst case is $k=\frac m2$, where the
estimate easily holds when $\t<\frac{n^3}{ 6m}$.

\bibliographystyle{amsplain}
 
\bibliography{Lmatrix}

\end{document}